\documentclass{amsart}

\usepackage{mathtools,amsfonts,amsthm,amssymb} 
\usepackage{enumitem}                   
\usepackage{tikz}                      
\usetikzlibrary{arrows.meta,chains,matrix,positioning,scopes,cd} 
\usepackage[utf8]{inputenc}            
\usepackage[T1]{fontenc}               
\usepackage{geometry}                  
\usepackage[pdftex,hidelinks]{hyperref} 

\newtheorem{theorem}{Theorem}

\newtheorem{proposition}[theorem]{Proposition}
\theoremstyle{definition}

\newtheorem{example}{Example}
\newtheorem{remark}{Remark}

\makeatother

\numberwithin{equation}{section} 

\begin{document}
\title[On Congruence Theorem for valued division algebras]{On Congruence Theorem for valued division algebras}

\author[Huynh Viet Khanh]{Huynh Viet Khanh$^\dagger$}
\address{$^\dagger$Department of Mathematics and Informatics, HCMC University of Education, 280 An Duong Vuong Str., Dist. 5, Ho Chi Minh City, Vietnam}
\email{khanhhv@hcmue.edu.vn}

\author[Nguyen Duc Anh Khoa]{Nguyen Duc Anh Khoa$^{\dagger\ddagger}$}
\address{$^\ddagger$Department of Mathematics, Le Hong Phong High School for the Gifted, 235 Nguyen Van Cu Str., Dist. 5, Ho Chi Minh City, Vietnam}
\email{anhkhoa27092002@gmail.com; khoanda.dais035@pg.hcmue.edu.vn}
\thanks{This research is funded by Ho Chi Minh City University of Education Foundation for Science and Technology under grant number CS.2023.19.55.}
	
\begin{abstract} 
	Let $K$ be a field equipped with a Henselian valuation, and let $D$ be a tame central division algebra over the field $K$. Denote by $\mathrm{TK}_1(D)$ the torsion subgroup of the Whitehead group ${\rm K}_1(D) = D^*/D'$, where $D^*$ is the multiplicative group of $D$ and $D'$ is its derived subgroup. Let ${\bf G}$ be the subgroup of $D^*$ such that $\mathrm{TK}_1(D) = {\bf G}/D'$. In this note, we prove that either $(1 + M_D) \cap {\bf G} \subseteq D'$, or the residue field $\overline{K}$ has characteristic $p > 0$ and the group ${\bf H} := ((1 + M_D) \cap {\bf G})D'/D'$ is a $p$-group. Additionally, we provide examples of valued division algebras with non-trivial ${\bf H}$. This illustrates that, in contrast to the reduced Whitehead group \({\rm SK}_1(D)\), a complete analogue of the Congruence Theorem does not hold for \({\rm TK}_1(D)\).
\end{abstract}
\maketitle

Let $D$ be a finite-dimensional division algebra with center $K$, and let ${\rm Nrd}_D: D \to K$ denote the reduced norm map. The \emph{Whitehead group} ${\rm K}_1(D)$ and the \emph{reduced Whitehead group} ${\rm SK}_1(D)$ of $D$ are defined as the quotient groups
\[
{\rm K}_1(D) = D^*/D' \quad \text{and} \quad {\rm SK}_1(D) = D^{(1)}/D',
\]
where $D^{(1)} = \{ x \in D \mid {\rm Nrd}_D(x) = 1 \}$ is the set of elements with reduced norm 1, and $D' = [D^*, D^*]$ is the derived group of $D^* = D \setminus \{0\}$. The study of \emph{reduced ${\rm K}_1$-theory} originated in 1943 with Nakayama and Matsushima (\cite{nakayama-matsushima-43}), who proved that ${\rm SK}_1(D)$ is trivial when $K$ is a $p$-adic field. For many years, it was conjectured that ${\rm SK}_1(D)$ is trivial for all division algebras, a question known as the Tannaka--Artin Problem. This problem remained open until 1975, when Platonov constructed the first example of a valued division algebra $D$ with non-trivial ${\rm SK}_1(D)$ (\cite{platonov-75}, \cite{platonov-75-2}). A cornerstone of Platonov's proof is the \emph{Congruence Theorem}, which establishes a connection between the reduced Whitehead groups ${\rm SK}_1(D)$ and ${\rm SK}_1(\overline{D})$, where $\overline{D}$ is the residue division algebra of $D$. The group ${\rm SK}_1(E)$ for a graded division algebra $E$ is extensively studied in Tignol and Wadsworth's book (\cite{tignol-wadsworth-10}). Their work provides detailed computations of ${\rm SK}_1(D)$ for valued division algebras, unifying and extending many key results (\cite[Chapter 11]{tignol-wadsworth-10}).

Let $D$ be a finite-dimensional division algebra with a Henselian center $K$. A valuation $v$ on a field $K$ is \emph{Henselian} if it extends uniquely to every finite field extension of $K$. As a result, $v$ extends uniquely to $D$, and we denote this extension by $v$. Let $V_D$ and $V_K$ denote the valuation rings of $v$ on $D$ and $K$, respectively, with maximal ideals $M_D$ and $M_K$. The residue division algebra and residue field are denoted by $\overline{D}$ and $\overline{K}$, respectively, and the value groups of $v$ on $D$ and $K$ are denoted by $\Gamma_D$ and $\Gamma_K$. The index of $D$, denoted ${\rm ind}(D)$, is defined as $\sqrt{[D:K]}$. With respect to this valuation, $D$ is \emph{unramified} over $K$ if $[\overline{D}:\overline{K}] = [D:K]$ and the center $Z(\overline{D})$ of $\overline{D}$ is separable over $\overline{K}$. At the other extreme, $D$ is \emph{totally ramified} over $K$ if $[\Gamma_D:\Gamma_K] = [D:K]$. We say $D$ is \emph{tame} if $Z(\overline{D})$ is separable over $\overline{K}$ and ${\rm char}(\overline{K})$ does not divide ${\rm ind}(D)/({\rm ind}(\overline{D})[Z(\overline{D}):\overline{K}])$. Furthermore, $D$ is \emph{strongly tame} if ${\rm char}(\overline{K})$ does not divide ${\rm ind}(D)$. Note that strong tameness implies tameness.

The Congruence Theorem asserts that for a tame division algebra $D$ over a Henselian center $K$, the intersection $(1 + M_D) \cap D^{(1)}$ is contained in the derived group $D'$. This theorem was first established by Platonov in 1975 for a complete, discrete valuation on $K$ (\cite{platonov-75}). However, Platonov's original proof was considerably lengthy and complicated. Subsequent works provided simpler alternative proofs (see, e.g., \cite{ershow-82}, \cite{hazrat-02}, \cite{suslin-91}). The theorem was later proven in full generality for any tame division algebra $D$ over a Henselian center by Hazrat and Wadsworth (\cite{hazrat-wadsworth-2011}).

Let $D$ be a finite-dimensional division algebra over its center $K$.  \textit{We define ${\rm K}_1(D)$ to be the torsion subgroup of the Whitehead group ${\rm K}_1(D)$}, consisting of the torsion elements of ${\rm TK}_1(D)$. Since the reduced Whitehead group ${\rm SK}_1(D)$ is ${\rm ind}(D)$-torsion, it is contained in ${\rm TK}_1(D)$. The reduced norm map ${\rm Nrd}_D: D \to K$ induces a homomorphism ${\rm Nrd}_D: {\rm TK}_1(D) \to \tau(K^*)$, where $\tau(K^*)$ denotes the torsion subgroup of the multiplicative group $K^* = K \setminus \{0\}$. As ${\rm SK}_1(D)$ is the kernel of this homomorphism, we obtain the isomorphism
\[
{\rm TK}_1(D)/{\rm SK}_1(D) \cong {\rm Nrd}_D({\rm TK}_1(D)) \subseteq \tau(K^*).
\]
In this note, we study the Congruence Theorem for ${\rm TK}_1(D)$ of a tame division algebra $D$ over a Henselian field $K$. The problem can be formulated as follows:

\begin{quote}
\emph{\textbf{Question.}} Let $K$ be a field with a Henselian valuation, and let $D$ be a tame $K$-central division algebra. Let ${\bf G}$ be the subgroup of $D^*$ such that ${\rm TK}_1(D) = {\bf G}/D'$. Does the inclusion $(1 + M_D) \cap {\bf G} \subseteq D'$ hold?
\end{quote}

\medskip 

The torsion Whitehead group ${\rm TK}_1(D)$ was first studied by Motiee in \cite{motiee-13}, where it was shown that the question posed earlier is answered affirmatively when $D$ is a strongly tame division algebra and ${\rm char}(\overline{K}) = {\rm char}(K)$. In this paper, we demonstrate that for a general tame valued division algebra, the answer is negative. Additionally, we provide computations for the group ${\bf H} := ((1 + M_D) \cap {\bf G})D'/D'$. By definition, ${\bf H} = 1$ if and only if $(1 + M_D) \cap {\bf G} \subseteq D'$.

Throughout this paper, we adopt the following notation. The groups ${\bf G}$ and ${\bf H}$, as previously defined, are used without further reference. For a group or ring $A$, let $A^*$ denote its multiplicative group and $Z(A)$ its center. For a group $G$, let $\tau(G)$ denote its torsion subgroup. For subgroups $H$ and $S$ of a group $G$, the subgroup $[H,S]$ is generated by all commutators $aba^{-1}b^{-1}$ with $a \in H$ and $b \in S$; in particular, $G' = [G,G]$ is the derived group of $G$. We also recall the definition of a cyclic algebra. Let $L/K$ be a cyclic  extension with Galois group ${\rm Gal}(L/K)$ generated by an automorphism $\sigma: L \to L$ of order $n = [L:K]$. For a non-zero element $a \in K$, we let
$$
D = L \oplus Lx \oplus Lx^2 \oplus \cdots \oplus Lx^{n-1},
$$
where multiplication  in $D$ is determined by the relations $x^n = a$ and $xb = \sigma(b)x$ for all $b \in L$. This algebra is denoted by $(L/K,\sigma,a)$, and is called the \textit{cyclic algebra associated with $(L/K,\sigma)$ and $a$} (see \cite[p. 218]{lam-01}). When $n$ is prime, $(L/K, \sigma, a)$ is a division algebra if and only if $a \notin N_{L/K}(L^*)$, where $N_{L/K}: L^* \to K^*$ is the field norm from $L$ to $K$ (\cite[Corollary 14.8]{lam-01}).

The main result of the this note is the following:
\begin{theorem}\label{Theorem_p-group}
	Let $K$ be a field with Henselian valuation, and let $D$ be a tame $K$-central division algebra. Then, either $(1+M_D) \cap {\bf G} \subseteq D'$ or ${\rm char}(\overline{K})=p>0$ and ${\bf H}$ is a $p$-group.
\end{theorem}
\begin{proof} 
	For each $x \in (1 + M_D) \cap {\bf G}$, let $k$ be the smallest positive integer such that $x^k\in D'$. It is enough to prove that $k$ must be a power of $p$. The proof will be finished  in two steps:
	
	\medskip 
	
	{\noindent \textit{\textbf{Step 1}}}. We prove that either $x\in D'$ or ${\rm char}(\overline{K})=p>0$ and $k$ divides $d$. Assume that $x \notin D'$. Then we have $k>1$. As $x \in 1+M_D$, there exist $m \in M_D$ such that $x = 1 + m$, and so $(1+m)^k \in D' \subseteq D^{(1)}$. As $D$ is a tame $K$-central division algebra, we get that ${\rm Nrd}_D\left(1+M_D\right) = 1+M_K$ (see \cite[Corollary 4.7]{hazrat-wadsworth-2011}). Hence, there exists $m_f \in M_K$ such that ${\rm Nrd}_D\left(1+m\right) = 1+m_f$. It follows that
	$$
	1={\rm Nrd}_D\left((1+m)^k\right) = {\rm Nrd}_D\left(1+m\right)^k = \left(1+m_f\right)^k.
	$$ 
	Thus, we have $1= \left(1+m_f\right)^k = 1+km_f + bm_f$, where $b \in M_K$. This implies that $(k\cdot 1+b)m_f = 0$, from which it follows that  $k\cdot 1 +b = 0$ or $m_f = 0$. If $m_f = 0$, we have ${\rm Nrd}_D\left(1+m\right) = 1$, implying $x \in D^{(1)}$. Therefore, we have $x \in D^{(1)} \cap \left(1+M_D\right) \subseteq D'$, which is a contradiction. In other words, we get $k\cdot 1 + b = 0$, which means $k \cdot \overline{1} = -\overline{b}$ in $\overline{K}$. As $b \in M_K$, we have $-\overline{b} = \overline{0}$, which means $k \cdot \overline{1} = 
	\overline{0}$, so ${\rm char}\left(\overline{K}\right) = p > 0$ and $p \mid  k$.
	
	\medskip 
		
	{\noindent \textit{\textbf{Step 2}}}. Assume that $p\mid k$. Write $k=p^mr$, where $(p,r)=1$. Assume by contradiction that $r>1$. As $p\nmid r$, we can repeat the arguments in Step 1 for $x^{p^m} \in D'$ instead of $x$ to get that $x^{p^m}\in D'$. This is a contradiction because $k$ was	chosen to be smallest.
\end{proof}

Thus, we obatin the following theorem:

\begin{theorem}[Congruence Theorem]\label{corollary_in_D'}
	Let $K$ be a field with Henselian valuation, and $D$ be a tame $K$-central division algebra. If ${\rm TK}_1(D)$ contains no elements of order ${\rm char}(\overline{K})$, then $(1+M_D) \cap {\bf G} \subseteq D'$.
\end{theorem}
\begin{proof}
	Assume by contradiction that $(1+M_D) \cap {\bf G} \nsubseteq D'$. It follows from previous theorem that ${\rm char}(K)=p>0$ and ${\bf H}$ is a non-trivial $p$-group. This implies that there exists a element $x\in {\bf G}$ such that  $x^{p^m}\in D'$, for some $m\geq 1$. Then $x^{p^{m-1}}D'$ is a non-trivial element of order $p$ of ${\rm TK}_1(D)$, a contradiction. 
\end{proof}

We present an example of a valued division algebra \(D\) with a non-trivial \({\bf H}\). This illustrates that, in contrast to \({\rm SK}_1(D)\), a complete analogue of the Congruence Theorem does not hold for \({\rm TK}_1(D)\).

\begin{proposition}\label{corollary_Qp}
	Let $(\mathbb{Q}_p,v_p)$ be the field of $p$-adic numbers with $p$ a prime, and $v_p$ the $p$-adic valuation on $\mathbb{Q}_p$. Let $D$ be a tame ${\mathbb{Q}_p}$-central division algebra with ${\rm ind}(D) = n$. Then, the following assertions hold:
	\begin{enumerate}[font=\normalfont]
		\item If $p>2$ or  $p=2$ and $n$ is odd, then ${\bf H}\cong \{\pm 1\}$.
		\item If $p=2$ and $n$ is even, then ${\bf H}=1$.
	\end{enumerate}
\end{proposition}
\begin{proof}
	We consider three possible cases:
	
	\medskip 
	
	\medskip 
	
	\textit{\textbf{Case 1}: $p>2$. } Then $\tau\left(\mathbb{Q}_p^*\right) = \left\{\varepsilon_1,\varepsilon_2,\ldots,\varepsilon_{p-1}\right\}$, where $\varepsilon_i$ is a $(p-1)$-th root of unity in $\mathbb{Q}_p$ for all $i \in \{1,2,\ldots,p-1\}$. As $\varepsilon_i^{p-1} = 1$, it follows that $\tau\left(\mathbb{Q}_p^*\right)$ contains no element of order $p$. Assume by contradiction that ${\bf H}$ has an element of order $p$, say $aD'$ where $a\in D$. Then, we have $a^p \in D'$ and $a \notin D'$. It follows that ${\rm Nrd}_D(a)^p= {\rm Nrd}_D\left(a^p\right) =1$, which implies that ${\rm Nrd}_D(a) \in \tau(\mathbb{Q}_p^*)$. Since $\tau(\mathbb{Q}_p^*)$ has no elements of order $p$, we get that ${\rm Nrd}_D(a) =1$, and hence $a \in D^{(1)}$. Because $a\in \left(1+M_D\right)$, we get $a \in \left(1+M_D\right)\cap {\bf G} \subseteq D'$, which is a contradiction. According to Theorem \ref{Theorem_p-group}, we conclude that ${\bf H}=1$.
	
	\medskip 
	
	\textit{\textbf{Case 2}: $p=2$ and $n$ is even.} Then $\tau\left(\mathbb{Q}_2^*\right) = \left\{\pm1\right\}$. As ${\rm TK}_1(D)/{\rm SK}_1(D)$ is isomorphic to $ {\rm Nrd}_D\left({\rm TK}_1(D)\right) \le \tau\left(\mathbb{Q}_2^*\right)$, we get that
	$$
	((1+M_D)\cap {\bf G})D^{(1)}/D^{(1)} \cong {\bf H}{\rm SK}_1(D)/{\rm 
		SK}_1(D)  \le 
	\tau\left(\mathbb{Q}_2^*\right).
	$$
	It follows that $((1+M_D)\cap {\bf G})D^{(1)}/D^{(1)} \subseteq \left\{-D^{(1)},D^{(1)}\right\}$. As ${\rm Nrd}_D(-1) = (-1)^n = 1$, we obtain $-D^{(1)} = D^{(1)}$. But we also have $-1 \in\left(1+M_D\right)\cap {\bf G}$, from which it follows that $((1+M_D)\cap {\bf G})D^{(1)}/D^{(1)} = 1$. Hence,
	$$
	\left(1+M_D\right)\cap {\bf G} \subseteq \left(1+M_D\right) \cap D^{(1)} \subseteq D',
	$$
	which implies that ${\bf H}=1$.
	
	\medskip 
	
	\textit{\textbf{Case 3}: $p=2$ and $n$ is odd.} As ${\rm Nrd}_D(-1) = (-1)^n = -1 \ne 1$, it follows that $D^{(1)} \ne -D^{(1)}$. On the other hand, we also have $-1 \in (1+M_D)\cap {\bf G}$. This means that $((1+M_D)\cap {\bf G})D^{(1)}/D^{(1)} = \{D^{(1)},-D^{(1)}\}$. Finally, we have 
	$$
	((1+M_D)\cap {\bf G})D^{(1)}/D^{(1)} \cong {\bf H}{\rm SK}_1(D)/{\rm SK}_1(D) \cong\tau\left(\mathbb{Q}_2^*\right),
	$$
	so ${\bf H} \cong \{\pm 1 \}$.
\end{proof}

We now give an example of a ${\mathbb{Q}_p}$-central division algebra satisfying Proposition \ref{corollary_Qp}. 

\begin{example}\label{example_1}
	Let $(\mathbb{Q}_p,v_p)$ be the field of $p$-adic numbers with $p$ a prime, and $v_p$ the $p$-adic valuation on $\mathbb{Q}_p$. Then $(\mathbb{Q}_p,v_p)$ becomes a Henselian valued field with $\overline{\mathbb{Q}_p}=\mathbb{F}_p$, the field of $p$ elements. Let $L$ be an unramified extension of $\mathbb{Q}_p$ of degree $n$. Then, with respect to the valuation $v$ extending the valuation $v_p$ on $\mathbb{Q}_p$, we have $L$ is cyclic Galois over $\mathbb{Q}_p$ as the valuation is Henselian and $\overline{L}=\mathbb{F}_{p^n}$ is cyclic Galois over $\overline{\mathbb{Q}_p}$. Let ${\rm Gal}(L/\mathbb{Q}_p)=\langle\sigma\rangle$. Take $\pi\in \mathbb{Q}_p$ such that $v(\pi)=1$. Let $D=(L/\mathbb{Q}_p, \sigma,\pi)$ be the cyclic algebra associated $L/\mathbb{Q}_p$ and $\pi$. Then, $Z(D) = \mathbb{Q}_p$ and $v$ extends to a valuation on $D$ given by 
	$$
	v\left(a_0+a_1x+\ldots+a_{n-1}x^{n-1}\right)=\min\left\{a_i+\dfrac{i}{n}\mid i \in 
	\{0;1;\ldots;n-1\}\right\}.
	$$
	Then, $D$ is tame $\mathbb{Q}_p$-central division algebra. Thus, all statements in Corollary \ref{corollary_Qp} hold for $D$.
\end{example}

The following example demonstrates that, although ${\bf H}$ is known to be a $p$-group, we still lose control of the order of ${\bf H}$. 
\begin{example}\label{example_2}
	Let $(\mathbb{Q}_p,v_p)$ be the field of $p$-adic numbers with $p$ a prime, and $v_p$ the $p$-adic valuation on $\mathbb{Q}_p$. Let $K$  be a field obtained from $\mathbb{Q}_p$ by adjoining all $p^k$-th roots of unity ($k\in\mathbb{N}$). Let us still denote by $v_p$ the unique extension of $v_p$ on $\mathbb{Q}_p$ to $K$. Then, we have $\overline{K}=\overline{\mathbb{Q}_p}=\mathbb{F}_p$. Let $L$ be an unramified extension of $K$ of degree $q$, where $q$ is a prime  number such that $(p,q) =(p-1,q) = 1$. Then, with respect to the valuation $v$ extending the valuation $v_p$ on $K$, the field $L$ is cyclic Galois over $K$ as the valuation is Henselian and $\overline{L}=\mathbb{F}_{p^q}$ is cyclic Galois over $\overline{K}$. Let ${\rm Gal}(L/K)=\langle\sigma\rangle$. Take $\pi\in K$ such that $v(\pi)=1$. Let $D=(L/K, \sigma,\pi)$ be the cyclic algebra associated $L/K$ and $\pi$. We claim that $D$ is a division ring. It suffices to prove that $\pi \notin N_{L/K}(L^*)$, where $N_{L/K}$ denotes the field norm from $L$ to $K$. Assume by contradiction that $\pi = N_{L/K}(a)$ for some $a \in L^*$. Then,
	$$
	\pi = N(a) = \sigma^{q-1}(a)\sigma^{q-2}(a)\ldots\sigma(a)a,
	$$
	which implies that $1=v(\pi)=v\left(\sigma^{q-1}(a)\sigma^{q-2}(a)\ldots\sigma(a)a\right) = qv(a)$, hence $v(a) = 1/q$. As $a\in L$ and $\Gamma_L = \Gamma_K$, we get $v(a)=1/q \in \Gamma_K$.	Let $A$ be the additive subgroup of $\Gamma_K$ generated by $1/q$.  As $\overline{K}=\overline{\mathbb{Q}_p}$, there exists a $p^k$-th root of unity $\varepsilon$, for some $k\in\mathbb{N}$, such that  $\mathbb{Q}_p<F:=\mathbb{Q}_p(\varepsilon)<K$ and
	$$
	|\Gamma_F/\mathbb{Z}| =\left|A/\mathbb{Z}\right| . \left|\Gamma_F/A
	\right|= 
	q . \left|\Gamma_F/A\right|.
	$$
	Then, we have  $[F:\mathbb{Q}_p]=[\mathbb{Q}_p(\varepsilon):\mathbb{Q}_p]=(p-1)p^{k-1}$. As $[F:\mathbb{Q}_p] = [\overline{F}:\overline{\mathbb{Q}_p}].|\Gamma_F/\Gamma_{\mathbb{Q}_p}| $ and $\Gamma_{\mathbb{Q}_p}=\mathbb{Z}$, we get $|\Gamma_F/\mathbb{Z}| = (p-1)p^{k-1}$ from  which it follows that $q\mid (p-1)p^{k-1}$, a contradiction Therefore, $\pi \notin 
	N_{L/K}(L^*)$, which means that $D$ is a division ring, as claimed. Moreover, $D$ is tame with the valuation defined in a similar way as we have done in previous example. 
	
	For each $k\in\mathbb{N}$, let $E=\{\varepsilon^r\mid p\nmid r<p^k, r\in \mathbb{N}\}$ be the set of all primitve $p^k$-th roots of unity. Then $E$ has $p^k-p^{k-1}$ which are roots of the polynomial
	$$
	f(x) = \dfrac{x^{p^k}-1}{x^{p^{k-1}}-1} = y^{p-1}+y^{p-2}+\ldots+1, \text{ where } y = x^{p^{k-1}}.
	$$
	Moreover, we have $f(x) = \prod\limits_{e \in E}(x-e)$. As $f(1) = p$, we have $p = \prod\limits_{e \in E}(1-e)$, which implies that 
	$$
	0<1 = v_p(p) = 
	\sum\limits_{e \in E} v_p(1-e).
	$$ 
	This shows that $1-e \in 
	M_D$, or equivalently, $e \in 1+M_D$ for all $e\in E$.  Next, we show that $1D',e^pD',e^{p^2}D',\ldots,e^{p^{k-1}}D'$ are distinct elements of ${\bf H}$. Indeed, assume that there exist $e^{p^i}$ and $e^{p^j}$ such that $e^{p^i}D' = 
	e^{p^j}D'$ ($0\le j<i \le k$), then $e^{p^i-p^j} \in D'$. It follows that 
	$$
	1={\rm Nrd}_D(e^{p^i-p^j}) = e^{q(p^i-p^j)}.
	$$
	As $e$ is a $p^{k}$-th root of unity, we get that $p^{k} \mid q(p^i-p^j)$. But, because $(p,q(p^{i-j}-1))= 1$, we conclude $p^k \mid p^j$, yielding that $k \le j$, which is a contradiction. Hence, the elements $1D',e^pD',e^{p^2}D',\ldots,e^{p^{k-1}}D'$ are distinct in ${\bf H}$. As $k$ was chosen arbitrarily, we conclude that ${\bf H}$ has infinitely many elements. 
\end{example}
\begin{remark}
	Determining the precise structure of the group ${\bf H}$ in Example \ref{example_2} remains an open problem. In view of Theorem \ref{Theorem_p-group}, it is natural to ask whether any given $p$-group can be realized as the group ${\bf H}$ for some  tame division algebra?
\end{remark}

We next explore applications of the Congruence Theorem in computing the group ${\rm TK}_1(D)$ for specific division algebras. The reduction map $V_D\to\overline{D}$ induces group homomorphism
$$
U_D\to\overline{D}^*\text{\;\;\;\; given by \;\;\;\;} a\mapsto a+M_D
$$
with kernel $1+M_D$. Let $g\in G$. Then $g^m\in D'$, for some $m$. By Wadsworth formula \cite{wadsworth-86}, we get 
$$
mv(g)=v(g^m)=\frac{1}{n}v({\rm Nrd}_D(g^m))=0.
$$
As $\Gamma_D$ is torsion-free, we get $v(g)=0$ and so $g\in U_D$. Thus, ${\bf G} \subseteq U_D$. Consider following diagram with exact rows:
\begin{center}
	\begin{tikzcd}
		1 \arrow{r} & (1+M_D)\cap D' \arrow{r} \arrow[hookrightarrow]{d} & D' \arrow[hookrightarrow]{d} \arrow{r} & \overline{D'}\arrow[hookrightarrow]{d} \arrow{r} & 1
		\\
		1 \arrow{r} & (1+M_D)\cap {\bf G}\arrow[rightarrow]{r} & {\bf G}\arrow{r} & \overline{{\bf G}} \arrow{r} & 1
	\end{tikzcd}.
\end{center}
If $(1+M_D)\cap D'=(1+M_D)\cap {\bf G}$, then we obtain the following isomorphism:
\begin{equation}\label{equation_isomorphic}
	{\rm TK}_1(D)\cong \overline{{\bf G}}/\overline{D'}.
\end{equation}

By applying \eqref{equation_isomorphic}, Motiee provided some computations of ${\rm TK}_1(D)$ in \cite[Theorem 10]{motiee-13} for certain division algebras. Employing the same method as in the proof of \cite[Theorem 10]{motiee-13}, we readily obtain the following proposition:

\begin{proposition}\label{propostion_4}
	Let $K$ be a field with Henselian valuation, and $D$ be a tame $K$-central algebra of index $n$. If ${\rm TK}_1(D)$ contains no elements of order ${\rm char}(\overline{K})$, then the following assertions hold:
	\begin{enumerate}[font=\normalfont]
		\item[(i)] If $D$ is unramified, then ${\rm TK}_1(D)\cong {\rm TK}_1(\overline{D})$.
		\item[(ii)] If $D$ is totally ramified, then ${\rm TK}_1(D)\cong \tau(\overline{K}^*)/\mu_e(\overline{K})$, where $e={\rm exp}(\Gamma_D/\Gamma_K)$ and $\mu_e(\overline{K})$ is the group of $e$-th unity in $\overline{K}$.
	\end{enumerate}
\end{proposition}
\begin{proof}
	We employ the method of Motiee as utilized in the proof of \cite[Theorem 10]{motiee-13}. According to Theorem \ref{corollary_in_D'}, we have $(1+M_D)\cap D'=(1+M_D)\cap {\bf G}$, and so (\ref{equation_isomorphic}) holds. Recall Ershov's formula (\cite[Corollary 2]{ershow-82}) that if $a\in U_D$ then
	\begin{equation}\label{equation_ershov}
	\overline{{\rm Nrd}_D(a)}=N_{Z(\overline{D})/\overline{K}}{\rm Nrd}_{\overline{D}}(\overline{a})^{n/{mm'}},
\end{equation}
	where $m={\rm ind}(\overline{D})$ and $m'=[Z(\overline{D}):\overline{K}]$.
	
	(i) Assume that $D$ is unramified. Let $\mathbb{G}$ be a subgroup of $\overline{D}^*$ such that ${\rm TK}_1(\overline{D})=\mathbb{G}/\overline{D}'$. Since $\Gamma_D=\Gamma_K$, it can be checked that $\overline{D'}=\overline{D}'$. Thus, by (\ref{equation_isomorphic}), we are done if we prove that $\overline{\bf G}=\mathbb{G}$. It is clear that $\overline{\bf G}\subseteq \mathbb{G}$.  For the converse, let $\overline{a}\in \mathbb{G}$. According to \cite[Theorem 3.2]{wad-2002}, we get that $[\overline{D}:\overline{K}]=[D:K]$ and $Z(\overline{D})=\overline{K}$. By equation (\ref{equation_ershov}), we have $\overline{{\rm Nrd}_D(a)}={\rm Nrd}_{\overline{D}}(\overline{a})=1$, which implies that $\overline{a}^m\in\overline{D}'$. It follows that $\overline{{\rm Nrd}_D(a)^m}={\rm Nrd}_{\overline{D}}(\overline{a}^m)=\overline{1}$, and so ${\rm Nrd}_D(a)^m\in 1+M_K={\rm Nrd}_D(1+M_D)$. Because ${\rm TK}_1(D)$ contains no elements of order ${\rm char}(K)$, we conclude that ${\rm char}(K)\nmid m$, and so by Hensel's Lemma, we get $ (1+M_K)^m= 1+M_K$. Thus, we get ${\rm Nrd}_D(a^m)\in {\rm Nrd}_D(1+M_D)^m$, from which it follows that there exist an element $b\in 1+M_D$ such that ${\rm Nrd}_D((ab^{-1})^m)=1$. This implies that $(ab^{-1})^m\in D^{(1)}$. As ${\rm SK}_1(D)$ is $n$-torsion, we get $(ab^{-1})^{mn}\in D'$, and thus, $ab\in{\bf G}$. It follows that $\overline{a}=\overline{ab}\in\mathbb{G}$, yielding that $\mathbb{G}\subseteq \overline{{\bf G}}$.
	
	(ii) The proof of this statement is identical to that of \cite[Theorem 10 (ii)]{motiee-13}, so we only provide the outline of the proof. First, it was proved in the proof of \cite[Theorem 10 (ii)]{motiee-13} that $\overline{{\bf G}}=\tau(\overline{F})$. Moreover, by \cite[Propostion 2.1]{hwang-wadsworth-99} that $\overline{D'}=\mu_e(\overline{F})$. Therefore (ii) follows from (\ref{equation_isomorphic}).
\end{proof}

\begin{remark}
	Suppose $D$ is a strongly tame division ring and ${\rm char}(K) = {\rm char}(\overline{K})$. By the Primary Decomposition Theorem (\cite[Theorem 5]{motiee-13}), it follows readily that ${\rm TK}_1(D)$ contains no elements of order ${\rm char}(\overline{K})$. However, the converse is not true. For instance, in Proposition \ref{corollary_Qp}, by appropriately choosing $n$ and $p$, we can construct a division ring $D$ such that ${\rm TK}_1(D)$ has no elements of order $p$. Therefore, Theorem \ref{corollary_in_D'} and Proposition \ref{propostion_4} generalize \cite[Theorems 9 and 10]{motiee-13}, respectively. (Note that the term ``tame'' in \cite[Theorem 9]{motiee-13} corresponds to ``strongly tame'' as used here.)
\end{remark}

{\noindent \textbf{Acknowledgements}. This research is funded by Ho Chi Minh City University of Education Foundation for Science and Technology under grant number CS.2023.19.55}.

\end{document}